\documentclass[a4paper,12pt]{article}

\usepackage[utf8]{inputenc}
\usepackage[T2A]{fontenc}
\usepackage[english, russian]{babel}
\usepackage{indentfirst}
\usepackage{hyphenat}

\usepackage{amsmath}
\usepackage{amssymb}
\usepackage{amsthm}
\usepackage{amsfonts}

\usepackage{geometry}
\newtheorem{theorem}{Теорема}

\newtheorem{definition}{Определение}

\title{О центральной предельной теореме для однородных нелинейных цепей Маркова в дискретном времени\footnote{Статья подготовлена в ходе проведения исследования в рамках Программы фундаментальных исследований Национального исследовательского университета «Высшая школа экономики» (НИУ ВШЭ). Работа выполнена при финансовой поддержке РФФИ, грант №20-01-00575a.}}

\author{Александр Щеголев\footnote{Национальный исследовательский университет <<Высшая школа экономики>>, Москва, ashchegolev@hse.ru}}

\begin{document}
\maketitle
\begin{abstract}
Класс нелинейных марковских процессов характеризуется наличием зависимости текущего состояния процесса от текущего распределения процесса в дополнение к зависимости от предыдущего состояния процесса. Благодаря этой особенности, данные процессы характеризуются сложным предельным поведением и эргодическими свойствами, для которых привычных критериев для марковских процессов недостаточно. Будучи разновидностью нелинейных марковских процессов, нелинейные цепи Маркова унаследовали эти особенности. В работе исследованы условия для выполнения центральной предельной теоремы для однородных нелинейных цепей Маркова в дискретном времени с конечным дискретным фазовым пространством. Также приведен краткий обзор известных результатов о эргодических свойствах нелинейных цепей Маркова. Полученный результат дополняет существующие результаты в данной области и может быть полезен для дальнейших приложений в статистике.
\medskip

\noindent{\em Ключевые слова:} нелинейные марковские цепи, центральная предельная теорема, экспоненциальная скорость сходимости
\end{abstract}

\section{Введение}
\label{S:1}
Класс случайных процессов, имеющих зависимость от закона распределения впервые был впервые предложен А.А.~Власовым \cite{Vlasov:1938} для описания динамики плазмы, представляющей собой большую систему заряженных частиц. В дальнейшем, общий класс процессов был изучен Г.П.~Маккином \cite{McKean:1966}. Данные процессы в литературе называют процессами Маккина-Власова или нелинейными марковскими процессами. Процессы этого вида естественным образом возникают при описании предельного поведения системы большого количества взаимодействующих частиц. Данный вид процессов изучался в работах множества авторов, среди которых можно выделить монографии А.С.~Шнитмана \cite{Sznitman:1991} и В.Н.~Колокольцова \cite{Kolokoltsov:2010}. В отличие от марковских процессов, нелинейные марковские процессы могут обладать сложными эргодическими свойствами \cite{Muzychka:2011,Neumann:2022}, так что классические результаты для марковских процессов могут быть неприменимы.

Изучению эргодических свойств нелинейных цепей Маркова также посвящен ряд работ.
Одни из первых результатов изучения эргодических свойств однородных нелинейных цепей Маркова в дискретном времени были получены О.А.~Бутковским \cite{Butkovsky:2014}. В работе были установлены условия существования и единственности инвариантной меры для нелинейных марковских цепей, а также показано на примерах, что стандартного результата для цепей Маркова (выполнения условия Маркова-Добрушина) в данном случае недостаточно для эргодичности процесса. В работе M.Х.~Сабурова \cite{Saburov:2016} рассматривается определенный класс (более узкий, чем в работе \cite{Butkovsky:2014}) нелинейных полиномиальных Марковских операторов на конечном пространстве, для которых показано, что при некоторых условиях на коэффициент эргодичности Добрушина достигается эргодичность. В работах \cite{Shchegolev:2021,Shchegolev:2022} также рассматриваются нелинейные цепи Маркова в дискретном времени, обобщается оценка \cite{Butkovsky:2014} для скорости сходимости на случай использования переходных вероятностей за несколько шагов. На примерах показано, что для переходных ядер за несколько шагов скорость сходимости может быть выше, и показан пример, демонстрирующий, что невыполнение условий \cite{Butkovsky:2014} в общем случае не препятствует сходимости к единственной инвариантной мере. Также в работе \cite{Shchegolev:2022} представлен результат об условиях выполнения закона больших чисел для нелинейных цепей Маркова в дискретном времени. В работе Б.А.~Нойман \cite{Neumann:2022} рассматриваются эргодические свойства нелинейных цепей Маркова с конечным пространством состояний в непрерывном времени, где показано, что для эргодичности нелинейной цепи Маркова с непрерывным временем должен существовать непрерывный нелинейный генератор.

В данной работе мы ограничиваемся рассмотрением класса нелинейных цепей Маркова в дискретном времени. Статья рассматривает условия выполнения центральной теоремы для нелинейных цепей Маркова, продолжая идею \cite{Shchegolev:2022}, где был рассмотрен соответствующий закон больших чисел. Полученный результат важен для статистических приложений и дальнейшего построения статистики.

\section{Постановка задачи и основной результат}
\label{S:2}

Прежде всего напомним определение нелинейной цепи Маркова в дискретном времени.

\begin{definition}
Нелинейной цепью Маркова c дискретным временем назовем процесс $\left(X_n^\mu\right)_{n\in\mathbb{Z_+}}$, принимающий значения на пространстве состояний $(E, \mathcal{E})$, который имеет начальное распределение $\mathcal{L}\left(X_0^\mu\right) = \mu$, $\mu \in \mathcal{P}(E)$ и переходные вероятности
\begin{align*}
P_{\mu_n}(x, B) =& \mathbb{P}\left(X_{n+1}^\mu \in B | X_n^\mu = x_n,\dots, X_{0}^\mu = x_{0};\right.&\\
&\left.\mathcal{L}(X^\mu_n) = \mu_n, \dots, \mathcal{L}(X^\mu_0) = \mu_0\right)&& \\
=& \mathbb{P}\left(X_{n+1}^\mu \in B | X_n^\mu = x; \mathcal{L}(X^\mu_{n-1}) = \mu_{n-1}\right),
\end{align*}
где $x_0,\dots,x_n \in E$, $B \in \mathcal{E}$, $n \in \mathbb{Z}_{+}$.
\end{definition}

Рассматриваемый далее результат устанавливает достаточные условия, чтобы для нелинейной марковской цепи с конечным дискретным фазовым пространством выполнялась центральная предельная теорема.

\begin{theorem}
Пусть нелинейная цепь Маркова $X_k^\mu$, определена на измеримом пространстве состояний $(E, \mathcal{E})$ и имеет единственную инвариантную меру, скорость сходимости к которой экспоненциальна. Обозначим через $X_k^\pi$ копию исходной нелинейной цепи Маркова $X_k^\mu$ с начальным распределением, равным инвариантной мере $\pi$, и переходными вероятностями $P_{x,y}^\pi$, равномерно отделенными от нуля,
\begin{equation*}
\inf_\mu \min_{x,x'} P^\mu_{x,x'} >0.    
\end{equation*}
Тогда для нелинейной цепи Маркова выполнена центральная предельная теорема,
\begin{align*}
\mathbb{E}\left[g\left(\frac{S_n}{\sqrt{n}}\right)\right] \rightarrow \mathbb{E}\left[g(\eta)\right],\quad \eta\sim\mathcal{N}(0, \sigma^2),
\end{align*}
где $g \in \mathcal{C}_b$,
\begin{gather*}
S_n = \sum_{i=0}^{n-1}\left(f(X^\mu_i)-{\mathbb{E}}[f(X^\pi_0)]\right),\\
\sigma^2 = \mathbb{E}_\pi\left[\sum_{0\le i\le j \le \infty}\left(f(X^\mu_i)-{\mathbb{E}}[f(X^\pi_0)]\right)\left(f(X^\mu_j)-{\mathbb{E}}[f(X^\pi_0)]\right)\right].
\end{gather*}
\label{nmcclt}
\end{theorem}

\begin{proof}
Прежде всего заметим, что поскольку функция $g$ ограничена, то для доказательства слабой сходимости достаточно рассмотреть неотрицательные $g\in \mathcal{C}_b$, такие что $g \ge \kappa$, где $\kappa$ -- некоторая положительная постоянная.

Также обратим внимание на то, что инвариантная копия исходного процесса, $X_k^\pi$, является однородной цепью Маркова, не зависящей от меры, а также обладает $\beta$-перемешиванием, благодаря равномерной экспоненциальной сходимости относительно начального распределения.

Рассмотрим последовательность сумм центрированных компонент для инвариантной копии исходного процесса, 
\begin{equation*}
S^{\pi}_n = \sum_{k=0}^{n-1}f(X^\pi_k)-{\mathbb{E}}[f(X^\pi_0)].
\end{equation*}

Тогда, в соответствии с результатами Ибрагимова и Линника \cite[Теорема~18.5.3]{Ibragimov:1965}, мы имеем, что 
\begin{equation*}
\mathbb{E}[\left|f(X^\pi_k)-{\mathbb{E}}[f(X^\pi_0)]\right|^{2+\zeta}]<\infty
\end{equation*}
 для некоторого $\zeta > 0$ и $\sum_{n=1}^{\infty}\alpha(n)^{\zeta/(2+\zeta)} <\infty$, где $\alpha(n)$ -- коэффициент перемешивания. Дисперсия случайной величины $S^\pi_n$, 
\begin{equation*}
\sigma^2 = \text{Var}(S^\pi_n) = \mathbb{E}_\pi\left[\sum_{0\le i\le j \le \infty}\left(f(X^\pi_i)-{\mathbb{E}}[f(X^\pi_0)]\right)\left(f(X^\pi_j)-{\mathbb{E}}[f(X^\pi_0)]\right)\right],
\end{equation*}
является конечной благодаря тому, что случайные величины $f(X^\pi_k)-{\mathbb{E}}[f(X^\pi_0)]$ удовлетворяют условию $\beta$-перемешивания, в данном случае, экспоненциального, и ограничены.

Тогда, последовательность $S^\pi_n$ удовлетворяет центральной предельной теореме, таким образом, $S^\pi_n \sim \mathcal{N}(0, \sigma^2/n)$. При этом случай $\sigma^2 = 0$ может быть интерпретирован как вырожденное нормальное распределение или $\delta$-мера Дирака, сконцентрированная в точке $0$.

Таким образом, центральная предельная теорема верна для нелинейной цепи Маркова, стартующей из инвариантного распределения, а именно
\begin{align}
\mathbb{E}\left[g\left(\frac{S^\pi_n}{\sqrt{n}}\right)\right] \to \mathbb{E}\left[g(\eta)\right],\quad \eta\sim\mathcal{N}(0,\sigma^2).
\label{clt_inv_nmc}
\end{align}
В случае $\sigma = 0$ получаем тождественное равенство обеих частей \eqref{clt_inv_nmc} величине $g(W)$, где $W$ -- некоторая константа.

Далее, покажем, что центральная предельная теорема верна и в более общем случае, то есть для нелинейной цепи Маркова, не находящейся изначально в стационарном состоянии.

Обозначим $\xi(X^\mu_i) = f(X^\mu_i)-\mathbb{E}[f(X_0^\pi)]$, тогда
\begin{flalign*}
\mathbb{E}\left[g\left(\frac{S_{n}}{\sqrt{n}}\right)\right] =&  \mathbb{E}\left[g\left(\frac{\sum_{i=0}^{n-1}\xi(X^\mu_i)}{\sqrt{n}}\right)\right] \\=&\iint\dots\int g\left(\frac{\sum_{i=0}^{n-1}\xi(X^\mu_i)}{\sqrt{n}}\right)\mu_0(x_0)P^{\mu_0}_{x_0,dx_1}P^{\mu_1}_{x_1,dx_2}\dots P^{\mu_{n-2}}_{x_{n-2},dx_{n-1}}\\
=&\iint\dots\int g\left(\frac{\xi(x_0)+\dots+\xi(x_{n-1})}{\sqrt{n}}\right)\mu_0(x_0)P^{\mu_0}_{x_0,dx_1}\dots P^{\mu_{n-2}}_{x_{n-2},dx_{n-1}}.
\end{flalign*}

Благодаря существованию единственной инвариантной меры и экспоненциальной скорости сходимости начальной меры к ней, имеем, $\|\mu_n-\pi\|_{TV} \le 2e^{-Cn}$, где $C$ -- некоторая константа, не зависящая от $n$. Исходя из предположения, что все элементы переходного ядра $P^{\pi}_{x,y}$ отделены от нуля, получаем нижнюю и верхнюю оценки:
\begin{equation*}
-\frac{\|\mu_n-\pi\|_{TV}}{2P^{\pi}_{x,y}} \le \frac{P^{\mu_n}_{x,y} - P^{\pi}_{x,y}}{P^{\pi}_{x,y}} \le \frac{\|\mu_n-\pi\|_{TV}}{2P^{\pi}_{x,y}},
\end{equation*}
откуда мы делаем вывод, что
\begin{equation}
1 - e^{\ln K - Cn} \le 1 - \frac{e^{-Cn}}{P^\pi_{x,y}} \le \frac{P^{\mu_n}_{x,y}}{P^{\pi}_{x,y}} \le 1 + \frac{e^{-Cn}}{P^\pi_{x,y}} \le 1 + e^{\ln K - Cn},
\label{Pnmc_Pmc}
\end{equation}
где $K$ -- некоторая константа, не зависящая от $n$.

Обозначим $\rho_g$ -- модуль непрерывности функции $g$ на замкнутом интервале $[-\|\xi\|, \|\xi\|]$ и выберем $\delta>0$ такое, что $\rho_g(\delta) < \kappa < g$. Также выберем $n$ по заданным $\delta$, $k$: $n \ge k \|\xi\|/\delta$, где $k\in\mathbb{Z}_+$, и, начиная с момента $k+1$, переходные ядра $P^{\mu_n}_{x,y}$ близки к инвариантному переходному ядру $P^{\pi}_{x,y}$ согласно \eqref{Pnmc_Pmc}. Определим переходное ядро за $k$ шагов как $Q^{\mu_0, k}_{x_0, dx_{k}}$, такое что
\begin{equation*}
\mu_{k} = \mu_0 Q^{\mu_0, k}_{x_{0}, dx_{k}}(dx_{k}) = \int \mu_0(dx_0)Q^{\mu_0, k}(x_0, dx_{k}).
\end{equation*}

Используя введенные обозначения, перепишем
\begin{flalign}
\mathbb{E}\left[g\left(\frac{S_{n}}{\sqrt{n}}\right)\right] =
\iint\dots\int g\left(\frac{\sum_{i=0}^k\xi(x_i)}{\sqrt{n}}+ \frac{\sum_{i=k+1}^{n-1}\xi(x_{i})}{\sqrt{n}}\right)
\mu_0Q^{\mu_0, k}(dx_{k})\prod_{j=k}^{n-2}P^{\mu_k}_{x_j,dx_{j+1}}\nonumber\\
\le \rho_g(\delta)+\prod_{j=k-1}^{n-2}\left(1 + e^{\ln K - Cj}\right)\left(\iint\dots\int g\left(\frac{\sum_{i=k+1}^{n-1}\xi(x_{i})}{\sqrt{n}}\right)\pi(dx_{k})\prod_{j=k}^{n-2}P^{\pi}_{x_j,dx_{j+1}}\right)
\label{ESn}
\end{flalign}

Распишем в явном виде аналогичное выражение для $S_n^\pi$:
\begin{flalign*}
\mathbb{E}\left[g\left(\frac{S^\pi_{n}}{\sqrt{n}}\right)\right] =&
\iint\dots\int g\left(\frac{\sum_{j=0}^{k}\xi(x_{j})}{\sqrt{n}}+ \frac{\sum_{j=k+1}^{n-1}\xi(x_{j})}{\sqrt{n}}\right)
\pi(dx_{k})\prod_{j=k}^{n-2}P^{\pi}_{x_j,dx_{j+1}}\\
\ge& -\rho_g(\delta)+\iint\dots\int g\left(\frac{\sum_{j=k+1}^{n-1}\xi(x_{j})}{\sqrt{n}}\right)\pi(dx_{k})\prod_{j=k}^{n-2}P^{\pi}_{x_j,dx_{j+1}}.
\end{flalign*}

Теперь мы можем подставить последний результат в формулу \eqref{ESn}, благодаря чему получаем
\begin{flalign*}
\mathbb{E}\left[g\left(\frac{S_{n}}{\sqrt{n}}\right)\right] &\le 
\rho_g(\delta) + \left(\rho_g(\delta) + \mathbb{E}\left[g\left(\frac{S^\pi_{n}}{\sqrt{n}}\right)\right]\right)\prod_{j=k-1}^{n-2}\left(1 + e^{\ln K - Cj}\right).
\end{flalign*}

Аналогичным образом мы можем вычислить оценку и для нижней границы $\mathbb{E}\left[g\left(\frac{S_{n}}{\sqrt{n}}\right)\right]$, получая
\begin{flalign*}
\mathbb{E}\left[g\left(\frac{S_{n}}{\sqrt{n}}\right)\right]
\ge& -\rho_g(\delta) + \prod_{j=k-1}^{n-2}\left(1-e^{\ln K - Cj}\right)\times\\
&\left(\iint\dots\int g\left(\frac{\sum_{i=k+1}^{n-1}\xi(x_i)}{\sqrt{n}}\right)
\pi(dx_{k})\prod_{j=k+1}^{n-2}P^{\pi}_{x_j,dx_{j+1}}\right)\\
\ge& -\rho_g(\delta) + \left(\mathbb{E}\left[g\left(\frac{S_n^\pi}{\sqrt{n}}\right)\right]-\rho_g(\delta)\right)\times \prod_{j=k-1}^{n-2}\left(1-e^{\ln K - Cj}\right).
\end{flalign*}

В результате, получаем границы для $\mathbb{E}\left[g\left(\frac{S_{n}}{\sqrt{n}}\right)\right]$,
\begin{equation*}
\begin{cases}
\displaystyle
\mathbb{E}\left[g\left(\frac{S_{n}}{\sqrt{n}}\right)\right] \ge -\rho_g(\delta) + \left(\mathbb{E}\left[g\left(\frac{S_n^\pi}{\sqrt{n}}\right)\right]-\rho_g(\delta)\right)\times \prod_{j=k-1}^{n-2}\left(1-e^{\ln K - Cj}\right),\\
\displaystyle
\mathbb{E}\left[g\left(\frac{S_{n}}{\sqrt{n}}\right)\right] \le
\rho_g(\delta) + \left(\rho_g(\delta) + \mathbb{E}\left[g\left(\frac{S^\pi_{n}}{\sqrt{n}}\right)\right]\right)\times\prod_{j=k-1}^{n-2}\left(1 + e^{\ln K - Cj}\right).
\end{cases}        
\end{equation*}

Для этих границ вычислим нижний и верхний пределы последовательности при $\delta \to 0$, $k \to \infty$ и $n \to \infty$, получая
\begin{gather*}
\liminf_{n\to\infty}\mathbb{E}\left[g\left(\frac{S_{n}}{\sqrt{n}}\right)\right] \ge \\
\liminf_{\delta\to 0}\liminf_{k\to\infty} \liminf_{n\to\infty}\left(
 -\rho_g(\delta) + \left(\mathbb{E}\left[g\left(\frac{S_n^\pi}{\sqrt{n}}\right)\right]-\rho_g(\delta)\right)\prod_{j=k-1}^{n-2}\left(1-e^{\ln K - Cj}\right)\right)
\end{gather*}
и
\begin{gather*}
\limsup_{n\to\infty}\mathbb{E}\left[g\left(\frac{S_{n}}{\sqrt{n}}\right)\right] \le \\
\limsup_{\delta\to 0} \limsup_{k\to\infty}\limsup_{n\to\infty}\left(
 \rho_g(\delta) + \left(\rho_g(\delta)+\mathbb{E}\left[g\left(\frac{S_n^\pi}{\sqrt{n}}\right)\right]\right)\prod_{j=k-1}^{n-2}\left(1+e^{\ln K - Cj}\right)\right).
\end{gather*}
Отметим, что при $\delta \to 0$, модуль непрерывности $\rho_g(\delta)$ стремится к нулю. Множители $\prod_{j={k}-1}^{n-2}\left(1-e^{\ln K - Cj}\right)$ и $\prod_{j={k}-1}^{n-2}\left(1+e^{\ln K - Cj}\right)$ при $k\to\infty$ равномерно стремятся к $1$ благодаря показателю в экспоненте. Используя результат \eqref{clt_inv_nmc}, мы получаем, что
для любой неотрицательной функции $g \in \mathcal{C}_b$, большей $\kappa$,
\begin{equation*}
\liminf_{n\to\infty} \mathbb{E}[g(S_n/\sqrt{n})] \ge  
\mathbb{E}\left[g(\eta)\right],
\end{equation*}
и
\begin{equation*}
\limsup_{n\to\infty} \mathbb{E}[g(S_n/\sqrt{n})] \le
\mathbb{E}\left[g(\eta)\right],
\end{equation*}
где $\eta$ -- гауссовская случайная величина с распределением $\mathcal{N}(0, \sigma^2)$.

Стало быть, эти неравенства имеют место и для любой функции $g\in \mathcal{C}_b$. Тогда по теореме о двух милиционерах, имеем
\begin{equation*}
\lim_{n\to\infty} \mathbb{E}[g(S_n/\sqrt{n})] =  
\mathbb{E}\left[g(\eta)\right],
\end{equation*}
что эквивалентно слабой сходимости, а следовательно, как и требовалось, центральная предельная теорема выполнена. Теорема \ref{nmcclt} доказана.
\end{proof}


\end{document}